\documentclass[12pt]{amsart}
\usepackage{amsmath,amssymb}
\usepackage[all]{xy}
\usepackage{pdfpages}
\usepackage[margin=1in]{geometry}
\usepackage{tikz}
\usepackage{tikz-3dplot}
\usetikzlibrary{decorations.pathreplacing, calligraphy}
\usepackage{scalerel}
\usepackage{hyperref}
\usepackage{tkz-euclide}
\usepackage{svg}
\usepackage{multirow}
\usepackage{caption}
\usepackage{subcaption}

\theoremstyle{plain}
\newtheorem{propn}{Proposition}[section]

\theoremstyle{definition}
\newtheorem*{maint}{Main Theorem}

\theoremstyle{remark}

\usepackage{hyperref} % Used to hyperlink url and email addresses
\hypersetup{pdfborder={0 0 0}, colorlinks=true, urlcolor=blue, citecolor=black}

\def\sideremark#1{\ifvmode\leavevmode\fi\vadjust{\vbox to0pt{\vss
\hbox to 0pt{\hskip\hsize\hskip1em%                          
\vbox{\hsize2cm\tiny\raggedright\pretolerance10000%        
\noindent {\color{red}{#1}}\hfill}\hss}\vbox to8pt{\vfil}\vss}}}%

\numberwithin{equation}{section}

\title{An optimization problem for triangles.}
\author{T. Murphy}
\address{Department of Mathematics, California State University Fullerton, 800 N. State College Blvd., Fullerton 92831 CA, USA.}
\email{tmurphy@fullerton.edu}
\author{K. Tran}
\begin{document}

\begin{abstract}

We consider the problem of optimizing the product of the distances from a given point in a triangle to each vertex. There are two possible cases in general. For isosceles triangles, we explicitly show exactly when both cases occur.  
\end{abstract}

\maketitle

\section{Introduction}
\subsection{Setup}
Let $\Delta$ denote a triangle in the plane determined by points $A$, $B$ and $C$. A classical problem is to optimize the function 
$$\zeta (p) := |pA| + |pB| + |pC|,$$
for $p\in \Delta$.   Fermat  posed this problem in a letter to Torricelli, the inventor of the barometer. Torricelli gave several solutions, including deriving a beautiful geometric construction proving that the  so-called \textit{Fermat point}  is the minimum value of $\zeta$. This point is constructed by building an equilateral triangle on each side of $\Delta$ and drawing the line through the outer vertex of each of these equilateral triangles and the midpoint of the corresponding edge on $\Delta$. These lines all meet in the Fermat point.  If the domain of $\zeta$ is restricted to  $\Delta$, then $\zeta$ achieves a maximum at the vertex opposite the shortest edge. This follows quite easily from the fact $\zeta$ is convex. For further information, see \cite{clark, courant, hajja, johnson}. 
  
  This note is concerned with a related question: how to  maximize  
  $$\varphi(p): = |pA||pB||pC|,
  $$
  for $p\in \Delta$.   This problem was posed  in the special case of equilateral triangles  by Finbarr Holland in some informal notes working through old Irish Mathematical Olympiad questions (\cite{holland}, Problem 105). Our theorem is a reasonably complete answer to the question. The strategy is to reduce to a one-dimensional problem and use calculus. 
  
  Since the question was posed in the realm of  Olympiad-style problems, it is desirable to seek a  proof using only synthetic or analytic methods. We could not get this approach to work but it remains an interesting question even in  the case of equilateral triangles. 
   
  \begin{maint}
  If $\Delta$ is acute, $\varphi$ has precisely three local maxima and three local minima.  If $\Delta$ is obtuse, then $\varphi$ has either (i)  three   local maxima and three local  minima, or  (ii) four local maxima and four local minima. 
  \end{maint}

 In the obtuse case, we show both cases can occur. In fact, we completely characterize all possibilities for  isosceles triangles in Section \ref{isoccase}. To explicitly answer  Holland's question, it follows from the Main Theorem  that for an equilateral triangle $\Delta$ with side length $l$  the maximum value of $\varphi$ is $\frac{\sqrt{3}}{8}l^3$, which is  attained on the midpoint of each side of $\Delta$.

 This problem is closely related to  classical work of Gauss on electrostatics. In particular, consider  point charges $m_i$ located at $z_i\in \mathbb{R}^2$, $1\leq i \leq n$.  Gauss studied the zeros of the conservative vector field generated by the electrostatic potential $\rho: \mathbb{R}^2\rightarrow \mathbb{R}$,  $$\rho(p) = \sum_{i=1}^n \frac{m_i}{\|p-z_i\|}.$$ 
 Consider the cases where there are three point charges ($n=3$) with all charges equal; $m_i=m$, $1\leq i \leq 3$. Then on $\Delta$, 
$\ln \rho = -m\ln\varphi$.  Since $\ln x$ is an increasing function, it easily follows that on $\Delta$ the functions $\rho$ and $\varphi$ have the same critical points.  In fact, Gauss proved all critical points of $\rho$ are saddle points \cite{marden}. Note $\rho$ is defined globally on $\mathbb{R}^2\backslash  \lbrace z_i, 1 \leq i \leq 3 \rbrace$, meaning the domains of $\rho$ and $\varphi$ differ. 
\bigskip

\textit{Acknowledgments} We thank R. Casper and  F. Holland for helpful conversations, and CSUF for supporting undergraduate research.

\section{ Initial Setup} 

It is immediate that minima of $\varphi$ occur at each vertex.
 Aside from these points, to optimize  $\varphi:  \Delta\rightarrow \mathbb{R}$ the first step is to observe $\ln\varphi$ is harmonic (away from arbitrarily small neighborhoods of each vertex). Shrinking these neighborhoods, applying the maximum principle and utilizing the fact that $\phi = \ln x$ is an increasing function immediately implies that $\varphi$ is extremized on the boundary $\partial \Delta$ of the triangle. The behaviour of $\varphi$ will be investigated on each side of $\Delta$ in turn. Take a fixed side of $\Delta$.  Choose coordinates so that the vertices of this side are at $(0,0)$ and $(1,0)$ and parametrize the side  as $(x,0): x \in [0,1]$. We will say we have chosen coordinates \textit{adapted} to this side.  Then  view $\varphi = \varphi(x)$ as a function of $x$.  The remaining vertex will have coordinates $(a,b)$ with $b>0$ without loss of generality.  Setting 
 \begin{equation}\label{e0}
 d_x = \sqrt{(x-a)^2 + b^2},
 \end{equation}
  we have   $\varphi(x) = x(1-x)d_x$ and 
\begin{equation}\label{e1}
\varphi'(x) = 0 \iff(1-2x)d_x^2 = x(x-1)(x-a)
\end{equation}

Equation (\ref{e1}) is generically a cubic equation.  As our function attains minima at both endpoints, this shows there are at most two maxima and one minima for $x\in(0,1)$ as well as a minima at each vertex. It is in fact possible to have two local maxima along an edge, as we will  show. So, this gives a crude a priori upper bound of twelve critical points for $\varphi$. 

Our choice of adapted coordinates to a given side of $\Delta$ will frequently rescale the triangle. Rescaling  by a factor $\lambda >0$  rescales the function $\varphi$ by a factor of $\lambda^3$, so the number and nature of critical points are not affected by this change.

\section{ The Isosceles case}\label{isoccase}

Set $\Delta$ to be isosceles with $AB$ = $BC$.  Beginning with the side $AC$, choose the corresponding adapted coordinate system with $(a,b)$ denoting the coordinates of $B$.  The isosceles condition implies that $a=1/2$ and we have seen that $x_0 = 1/2$ is a critical point of $\varphi$. For all other points  $x\in (0,1)$,   Equation (\ref{e1}) readily implies that  $x_1\neq 1/2$ is a critical point of $\varphi$ if and only if $2d_{x_1}^2 = x_1-x_1^2$. By Equation (\ref{e0}), this  is equivalent to 
\begin{equation}\label{e2}
x_1 = \frac{1}{2}\pm \sqrt{\frac{1}{4} -\frac{1}{6}\bigg(1+4b^2\bigg)}\bigg) 
\end{equation}
In particular, if $b> \frac{1}{\sqrt{8}}$ there are no other critical points. In such cases $x_0=1/2$ must be the unique critical point, and hence it is a  maximum. This corresponds to $\theta >\arctan (\frac{1}{\sqrt{2}})$. By the way, this fully answers the original question of Holland. For equilateral triangles  $\theta = \pi/3>  \arctan (\frac{1}{\sqrt{2}})$ and so there are exactly three minima (one at each vertex) and three maxima (at the midpoints of each side of the triangle).  The maximum value of $\frac{\sqrt{3}}{8}l^3$ mentioned in the introduction can be directly calculated from this.  When $b = \frac{1}{\sqrt{8}}$  no new critical point arises from Equation (\ref{e2}).  If $b< \frac{1}{\sqrt{8}}$, Equation (\ref{e2}) shows that there are exactly two more critical points labeled $x_1$, $x_2$, with $x_1<x_2$.  It is easy to check $x_1, x_2 \in (0,1)$, and that at   $\tilde{x}=x_1, x_2$ one has
$$
\varphi''(\tilde{x})= -3d^{-1}_{\tilde{x}}\left(1-2\tilde{x}\right)^2<0. 
$$

We conclude there are two local maxima at $x_1,  x_2 \in (0,1)$, and hence the critical point at $x_0=1/2$ must be a local minima.  This precisely explains when case (ii) of the Main Theorem occurs for an isosceles triangle, as will be clear presently.

It remains to analyze the remaining two sides of $\Delta$. Choose now coordinates adapted to the side $AB$. The case of the side $BC$ is completely analogous. Here $C$ has coordinates  $(a,b)$ with  $0<a<1$. The isosceles condition $|AB| = |BC|$ implies that
\begin{equation}\label{e3}
b^2=2a-a^2.
\end{equation}
If there were a critical point on this side at $x_0=\frac{1}{2}$, then Equation (\ref{e1}) would imply that $a=\frac{1}{2}$, i.e. $\Delta$ is equilateral. Otherwise, with $x_1\neq \frac{1}{2}$,  one can simply substitute Equation \ref{e3} into the  $d_x$ term in Equation (\ref{e1}) to get that $\tilde{x}$ is a critical point if, and only if,  
$$
\tilde{x}^2 - 2a\tilde{x} + 2a = (\tilde{x}-a)^2 + b^2 = d_{\tilde{x}}^2 = \frac{(\tilde{x}-1)(\tilde{x})(\tilde{x}-a)}{1-2\tilde{x}}
$$
This is equivalent to $\tilde{x}$ being a root of the cubic
\begin{equation}\label{isoscubic}
g(x) = x^{3}-\frac{5a+1}{3}x^{2}+\frac{7ax}{3}-\frac{2a}{3},
\end{equation}
where $a\in (0,1)$. Applying the Intermediate Value Theorem, there is always a root in $(0,1)$. We claim in fact there is a unique root. Assume to the contrary there are at least two roots. Then $g$ has a local maximum and minimum in $(0,1)$, or equivalently $g'(x)$ has two roots. Computing the discriminant of this quadratic, we see this is the case precisely when
\begin{equation}\label{avalue}
25a^2-53a+1 >0 \iff a < \frac{53-3\sqrt{301}}{50} \approx 0.01903
\end{equation}
Moreover, the smaller critical point 
$$
x_1 =\frac{1}{9}\left(5a+1-\sqrt{25a^{2}-53a+1}\right)
$$
must be the local maximum,  because $f$ is a cubic with positive leading term. 
To derive a contradiction, a straightforward algebraic computation show that 
\begin{equation}
f(x_1) = \frac{-250a^{3}+795a^{2}-327a-2+(50a^{2}-106a+2)\sqrt{25a^{2}-53a+1}}{729}.
\end{equation}
We claim that $f(x_1)<0$ for any $a$ satisfying Equation (\ref{avalue}). This follows from the following simple observations for $a$ satisfying Equation (\ref{avalue}); firstly that  $795a^{2}-327a<0$, and also that   $25a^{2}-53a+1 <1$.
The first observation ensures the  positive term $795a^{2}$ is outweighed by the negative term $-327a$. The second one implies that the  positive term $$(50a^{2}-106a+2)\sqrt{25a^{2}-53a+1}$$ is outweighed by the negative term $-2$. This establishes the claim, and shows there is a unique root of the cubic in $(0,1)$.

In summary, there is always  a unique maxima on the interiors of $AB$ and $BC$, and we have either one maxima or two maxima and one minima on the interior of $AB$ (depending on whether or not $b> \frac{1}{\sqrt{8}}$). Taking into account the vertices, this establishes our Theorem for isosceles triangles and moreover precisely delineates which case we are in.

\section{The general case}
 As before, the idea is to  analyse each side in turn and use calculus. For the cases where $\Delta ABC$ has an obtuse angle, we need the following result. We remind the reader the vertex opposite the line in question has coordinates $(a,b)$ with $b>0$ in our adapted coordinate system.  Throughout, set $\phi = \ln \varphi$.

\begin{propn}\label{obtuseopposite}
  If $a < 0$ or $a > 1$, $\varphi$ has exactly one critical point in $(0,1)$. 
\end{propn}

\begin{proof}
  
    \begin{equation}\label{second-derivative-bound}
        \begin{split}
            \phi''(x) & = \frac{b^2 - (x-a)^2}{((x-a)^2 + b^2)^2} - \left( \frac{1}{x^2} + \frac{1}{(1-x)^2} \right) \\
                      & < \frac{b^2 + (x-a)^2}{((x-a)^2 + b^2)^2} - \left( \frac{1}{x^2} + \frac{1}{(1-x)^2} \right) \\
                      & = \frac{1}{(x-a)^2 + b^2} - \left( \frac{1}{x^2} + \frac{1}{(1-x)^2} \right) \\
                      & < \frac{1}{(x-a)^2} - \left( \frac{1}{x^2} + \frac{1}{(1-x)^2} \right). \\
        \end{split}
    \end{equation}

    If $a < 0$, we get:
    \begin{displaymath}
        \begin{split}
            \phi''(x) & < \frac{1}{(x-a)^2} - \left( \frac{1}{x^2} + \frac{1}{(1-x)^2} \right) \\
                      & < \frac{1}{x^2} - \left( \frac{1}{x^2} + \frac{1}{(1-x)^2} \right) \\
                      & = -\frac{1}{(1-x)^2} < 0.
        \end{split}
    \end{displaymath}
    
    When $a>1$, we can apply a reflection around the axis  $y=\frac{1}{2}$, which is an isometry, to reduce the problem to the previous case, or one can directly compute in an analogous manner.  In all cases, $\phi''$ is negative so $\phi'$ is strictly decreasing. Since $\phi'(x) \rightarrow \infty$ as $x \rightarrow 0$ and $\phi'(x) \rightarrow -\infty$ as $x \rightarrow 1$, $\phi$ has exactly one maximum along $x \in (0,1)$ when $a < 0$ or $a > 1$. Furthermore, $\phi = \ln\varphi$ share the same maxima with $\varphi$, so $\varphi$ also has exactly one maximum along $x \in (0,1)$.
\end{proof}

The remaining ingredient is the following result. Let $\theta$ denote the angle at the vertex opposite the side we are studying. 
 
\begin{propn}\label{allacute}
 If $0 < a < 1$ and $\theta<\frac{\pi}{2}$, then $\varphi$ has exactly one critical point in $(0,1)$.
\end{propn}

\begin{proof}

    First, we rewrite the condition $\theta<\frac{\pi}{2}$ as an inequality involving $a$ and $b$. The vectors from $(a,b)$ to $(0,0)$ and $(1,0)$ are $(-a,-b)$ and $(1-a,-b)$, so $\theta<\frac{\pi}{2}$ if and only if their dot product is positive. A simple computation shows this is equvalent to 
    \begin{equation}\label{keyineq}
  b^2 > a(1-a).
    \end{equation}

    Equation (\ref{second-derivative-bound}) shows that 
    \begin{displaymath}
        \phi''(x) \leq \frac{1}{(x-a)^2 + b^2} - \left( \frac{1}{x^2} + \frac{1}{(1-x)^2} \right).
    \end{displaymath}
    Therefore, it suffices to show that
    \begin{equation}\label{e4}
        \frac{1}{(x-a)^2 + b^2} < \frac{1}{x^2} + \frac{1}{(1-x)^2} 
    \end{equation}
for $0 < x < 1$.    Let $u = x^2$, $v = (1-x)^2$, and $S = u + v = x^2 + (1-x)^2 > 0$, so 
    \begin{equation}\label{e4.5}
        \frac{1}{x^2} + \frac{1}{(1-x)^2} = \frac{u + v}{uv} = \frac{S}{uv}.
    \end{equation}
 Hence, after multiplying by the positive quantity $uv \bigl( (x-a)^2 + b^2 \bigr)$,  Equation (\ref{e4}) is equivalent  to
    \begin{equation}\label{e5}
        \bigl( (x-a)^2 + b^2 \bigr) S > uv.
    \end{equation}
 Note that
    \begin{displaymath}
        (x-a)^2 + b^2 = \bigl( (x-a)^2 + a(1-a) \bigr) + \bigl( b^2 - a(1-a) \bigr),
    \end{displaymath}
    and
    \begin{displaymath}
    (x-a)^2 + a(1-a)  =(1-a)u + av. 
 \end{displaymath}
 Combine these two observations with Equation (\ref{e4.5}) to obtain
    \begin{displaymath}
        \begin{split}
            \bigl( (x-a)^2 + b^2 \bigr) S - uv
            & = \Bigl( \bigl( (1-a)u + av \bigr) (u+v) - uv \Bigr) + \bigl( b^2 - a(1-a) \bigr) S \\
            & = (1-a)u^2 + av^2 + \bigl( b^2 - a(1-a) \bigr) S > 0,
        \end{split}
    \end{displaymath}
    since we know that $0 < a < 1$, $b^2 - a(1-a) > 0$ by Equation (\ref{keyineq}), and $u$, $v$, and $S$ are all positive. Hence $\phi''(x) < 0$ for all $0 < x < 1$, so $\phi'$ is strictly decreasing and crosses the $x$-axis exactly once. Hence  Equation (\ref{e5}) is established, implying that $\varphi$ has exactly one critical point in $(0,1)$ as required.
\end{proof}

Combing all these observations implies the main theorem as follows. 

\begin{proof}  If the triangle is acute, apply Proposition \ref{allacute} to each side in turn to yield the result. If the triangle is obtuse,  use Proposition \ref{obtuseopposite} twice for the two sides which form the obtuse angle. For the side opposite this obtuse angle, use  Equation (\ref{e1}) to give a crude answer of three critical points on this side. As we saw in the analysis of the isosceles case, if there are three critical points on the interior of  a given side, it could happen that one is a local minimum and two are local maxima. In general, one also has to consider the various possibilities encompassing the possibility of  points of inflection. This simple analysis is left to the reader; in all cases there is either a unique maximum and the two minima at each vertex, or the case we have discussed.  If there are two critical points on the interior of our side, one must be a maximum, and it follows the other is a point of inflection. Adding in the minima at each vertex and combining all this  gives us either three or four local maxima and  local minima on the boundary of $\Delta$, as claimed.  \end{proof}

\end{document}